\documentclass[12pt,reqno]{article}

\usepackage{amsmath,amsthm,amsfonts,amssymb,latexsym,mathrsfs,color}
\usepackage[colorlinks=true,
linkcolor=webblue, filecolor=webbrown, citecolor=webred ]{hyperref}
\definecolor{webblue}{rgb}{0,.5,0}
\definecolor{webred}{rgb}{0,.5,0}
\definecolor{webbrown}{rgb}{.6,0,0}

\newtheorem{thm}{Theorem}[section]

\newtheorem{cor}[thm]{Corollary}
\newtheorem{prop}[thm]{Proposition}

\theoremstyle{definition}

\newtheorem{rem}[thm]{Remark}

\numberwithin{equation}{section}

\def\la{\lambda}

\title{Proofs of some conjectures on monotonicity of number-theoretic and combinatorial sequences
\thanks{{\it Email addresses:} wangyi@dlut.edu.cn (Y. Wang), zhubaoxuan@yahoo.com.cn (B.-X. Zhu)}}
\author{Yi Wang$^1$,\quad Bao-Xuan Zhu$^2$}
\date{\footnotesize 1. School of Mathematical Sciences, Dalian University of Technology, Dalian 116024, PR China\\
2. School of Mathematical Sciences, Jiangsu Normal University, Xuzhou 221116, PR China}

\begin{document}

\maketitle

\begin{abstract}
We develop techniques to deal with monotonicity of sequences
$\{z_{n+1}/z_n\}$ and $\{\sqrt[n]{z_n}\}$.
A series of conjectures of Zhi-Wei Sun and of Amdeberhan {\it et al.} are verified in certain unified approaches.
\end{abstract}
%%%%%%%%%%%%%%%%%%%%%%%%%%%%%%%%%%%%%%%%%%%
\section{Introduction}
%%%%%%%%%%%%%%%%%%%%%%%%%%%%%%%%%%%%%%%%%%%
Let $p_n$ denote the $n$th prime.
In 1982, F. Firoozbakht conjectured that the sequence $\{\sqrt[n]{p_n}\}_{n\ge 1}$ is strictly decreasing,
which has been confirmed for $n$ up to $4\times 10^{18}$.
This conjecture implies the inequality $p_{n+1}-p_n<\log^2{p_n}-\log{p_n}$ for large $n$,
which is even stronger than Cram\'er's conjecture $p_{n+1}-p_n=O(\log^2{p_n})$.
See Sun~\cite{Sun-BAMS} for details.
Motivated by this, Sun~\cite{Sun-conj} posed a series of conjectures about monotonicity of sequences of the form $\{\sqrt[n]{z_n}\}_{n\ge 1}$,
where $\{z_n\}_{n\ge 0}$ is a familiar number-theoretic or combinatorial sequence.
Now partial progress has been made,
including Chen {\it et al.}~\cite{CGW} for the Bernoulli numbers,
Hou {\it et al.}~\cite{HSW12} for the Fibonacci numbers and derangements numbers,
Luca and St\u{a}nic\u{a}~\cite{LS12} for the Bernoulli, Tangent and Euler numbers.
The main object of this paper is to develop techniques to deal with monotonicity of $\{z_{n+1}/z_n\}$ and $\{\sqrt[n]{z_n}\}$ in certain unified approaches.

Two concepts closely related to monotonicity are log-convexity and log-concavity.
Let $\{z_n\}_{n\geq0}$ be a sequence of positive numbers.
It is called {\it log-convex} if $z_{n-1}z_{n+1}\ge z_n^2$ for all $n\ge 1$
and {\it strictly log-convex} if the inequality is strict.
The sequence is called {\it log-concave} if the inequality changes its direction.
Clearly, a sequence $\{z_n\}_{n\ge 0}$ is log-convex (log-concave, resp.)
if and only if the sequence $\{z_{n+1}/z_n\}_{n\ge 0}$ is increasing (decreasing, resp.).
The log-convex and log-concave sequences arise often in combinatorics, algebra, geometry, analysis, probability and
statistics and have been extensively investigated.
%Many well-known number-theoretic or combinatorial sequences are log-convex or log-concave.
We refer the reader to \cite{Sta89,Bre94, WY07} for log-concavity and \cite{LW07,Zhu12} for log-convexity.

On the other hand,
there are certain natural links between these two sequences $\{z_{n+1}/z_n\}$ and $\{\sqrt[n]{z_n}\}$.
For example, it is well known that if the sequence $\{z_{n+1}/z_n\}$ is convergent,
then so is the sequence $\{\sqrt[n]{z_n}\}$.
There is a similar result for monotonicity:
if the sequence $\{z_{n+1}/z_n\}$ is increasing (decreasing),
then so is the sequence $\{\sqrt[n]{z_n}\}$ when $z_0\le 1$ ($z_0\ge 1$).
See Theorem~\ref{thm+inc} for the details.
Thus we may concentrate our attention on log-convexity and log-concavity of sequences.

In the next section
we present our main results about monotonicity of $\{z_{n+1}/z_n\}$ and $\{\sqrt[n]{z_n}\}$.
As applications
we verify some conjectures of Sun~\cite{Sun-conj} and Amdeberhan~{\it et al.}~\cite{AMV12}.
In Section 3 we propose a couple of problems for further work.
%%%%%%%%%%%%%%%%%%%%%%%%%%%%%%%%%%%%%%%%%%%
\section{Theorems and applications}
%%%%%%%%%%%%%%%%%%%%%%%%%%%%%%%%%%%%%%%%%%%
We first show that the monotonicity of $\{z_{n+1}/z_n\}$ implies that of $\{\sqrt[n]{z_n}\}$.

\begin{thm}\label{thm+inc}
Let $\{z_n\}_{n\ge 0}$ be a sequence of positive numbers.
\begin{itemize}
\item [\rm (i)] Assume that $\{z_n\}_{n\ge 0}$ is log-convex.
If $z_0\le 1$ (and $z_{1}^2<z_{0}z_{2}$), then the sequence $\{\sqrt[n]{z_n}\}_{n\ge 1}$ is (strictly) increasing.
\item [\rm (ii)] Assume that $\{z_n\}_{n\ge 0}$ is log-concave and $z_0\ge 1$.
Then the sequence $\{\sqrt[n]{z_n}\}_{n\ge 1}$ is decreasing.
If $z_0>1$ or $z_{1}^2>z_{0}z_{2}$, then $\{\sqrt[n]{z_n}\}_{n\ge 1}$ is strictly decreasing.
\item [\rm (iii)] Assume that $\{z_n\}_{n\ge N}$ is log-convex
and $\sqrt[N]{z_N}<\sqrt[N+1]{z_{N+1}}$ for some $N\ge 1$.
Then $\{\sqrt[n]{z_n}\}_{n\ge N}$ is strictly increasing.
The similar result holds for log-concave sequences.
\end{itemize}
\end{thm}
\begin{proof}
(i)\quad
Let $x_n=z_n/z_{n-1}$ for $n\ge 1$.
Then by the log-convexity of $\{z_n\}$,
the sequence $\{x_n\}$ is increasing:
\begin{equation}\label{xni}
x_1\le x_2\le \ldots\le x_n\le x_{n+1}\le\ldots.
\end{equation}
Write
$$z_n=\frac{z_n}{z_{n-1}}\frac{z_{n-1}}{z_{n-2}}\ldots \frac{z_1}{z_0}z_0=x_nx_{n-1}\ldots x_1z_0.$$
Then
\begin{equation}\label{nan}
\frac{\sqrt[n+1]{z_{n+1}}}{\sqrt[n]{z_n}}=\frac{\sqrt[n+1]{x_{n+1}x_nx_{n-1}\ldots x_1z_0}}{\sqrt[n]{x_nx_{n-1}\ldots x_1z_0}}
=\frac{\sqrt[n+1]{x_{n+1}}}{\sqrt[n(n+1)]{x_nx_{n-1}\ldots x_1z_0}}\ge \sqrt[n+1]{\frac{x_{n+1}}{x_{n}}}\ge 1
\end{equation}
since $z_0\le 1$ and (\ref{xni}).
Thus the sequence $\{\sqrt[n]{z_{n}}\}_{n\ge 1}$ is increasing.
Clearly, if $z_1^2<z_0z_2$, i.e., $x_1<x_2$, then the first inequality in (\ref{nan}) is strict,
and so the sequence $\{\sqrt[n]{z_{n}}\}_{n\ge 1}$ is strictly increasing.

(ii)\quad
Note that a sequence $\{z_n\}$ is log-concave if and only if the sequence $\{1/z_n\}$ is log-convex.
Hence (ii) can be proved as did in (i).

(iii)\quad Let $y_n=z_N^{N+1-n}/z_{N+1}^{N-n}$ for $0\le n\le N-1$ and $y_n=z_n$ for $n\ge N$.
Then (iii) follows by applying (i) and (ii) to the sequence $\{y_n\}_{n\ge 0}$ respectively.
\end{proof}

\begin{rem}
(A)\quad
Although $\{n\}_{n\ge 0}$ is log-concave,
$\{\sqrt[n]{n}\}_{n\ge 1}$ is not decreasing since $\sqrt{2}=\sqrt[4]{4}$.
However, $\{\sqrt[n]{n}\}_{n\ge 3}$ is decreasing by Theorem~\ref{thm+inc}~(iii).

(B)\quad
We can replace the condition $\sqrt[N]{z_N}<\sqrt[N+1]{z_{N+1}}$ ($\sqrt[N]{z_N}>\sqrt[N+1]{z_{N+1}}$, resp.)
by $z_N^2<z_{N+1}$ ($z_N^2>z_{N+1}$, resp.) in Theorem~\ref{thm+inc}~(iii).
In this case we define $y_n=1$ for $0\le n\le N-1$ and $y_n=z_n$ for $n\ge N$.

(C)\quad
It is possible that $\{\sqrt[n]{z_n}\}_{n\ge 1}$ is monotonic but $\{z_{n+1}/z_n\}_{n\ge 0}$ is not.
For example, let $F_n$ be the $n$th Fibonacci number:
$F_0=0, F_1=1$ and $F_{n+1}=F_{n-1}+F_n$.
It is showed that $\{\sqrt[n]{F_n}\}_{n\ge 2}$ is strictly increasing~\cite[Theorem 1.1]{HSW12}.
However, $\{F_n\}$ is neither log-concave nor log-convex since $F_{n-1}F_{n+1}-F^2_n=(-1)^{n}$ for $n\ge 2$.
\end{rem}

We next apply Theorem~\ref{thm+inc} to verify some conjectures posed by Sun in~\cite{Sun-conj}.

The Bell number $B(n)$ counts the number of partitions of the set $\{1,\ldots, n\}$ into disjoint nonempty subsets.
It is known that
$$\{B(n)\}_{n\ge 0}=\{1,1,2,5,15,52,203,877,\ldots\}.\qquad \cite[A000110]{OEIS}$$
Engel~\cite{Eng94} showed that the sequence $\{B(n)\}$ is log-convex.
So by Theorem~\ref{thm+inc} we have the following result, which was conjectured by Sun~\cite[Conjecture 3.2]{Sun-conj}.

\begin{cor}
The sequence $\{\sqrt[n]{B(n)}\}_{n\ge 1}$ is strictly increasing.
\end{cor}

Let $p(n)$ denote the number of partitions of a positive integer $n$.
Then
$$\{p(n)\}_{n\ge 1}=\{1,2,3,5,7,11,15,22,30,\ldots\}.\qquad\cite[A000041]{OEIS}$$
Janoski~\cite{Jan12} showed the sequence $\{p(n)\}_{n\ge 25}$ is log-concave,
which was conjectured by Chen~\cite{Che10}.
Note that $p(25)=1958$ and $p(26)=2436$.
It follows from Theorem~\ref{thm+inc}~(iii) that $\{\sqrt[n]{p(n)}\}_{n\ge 25}$ is strictly decreasing.
Thus we have the following result
(the case for $6\le n\le 24$ may be confirmed directly),
which was conjectured by Sun~\cite[Conjecture 2.14]{Sun-conj}.

\begin{cor}
The sequence $\{\sqrt[n]{p(n)}\}_{n\ge 6}$ is strictly decreasing.
\end{cor}

Many combinatorial sequences satisfy a three-term recurrence.
Do\v{s}li\'c~\cite{Dos10}, Liu and Wang~\cite{LW07} gave some sufficient conditions for log-convexity of such sequences.
The following result is a variation of Liu and Wang~\cite[Theorem 3.1]{LW07}.

\begin{prop}\label{lw-3.1}
Let $\{z_n\}_{n\ge 0}$ be a sequence of positive numbers and satisfy %the recurrence
\begin{equation}\label{eq-rr+}
a_nz_{n+1}=b_nz_n+c_nz_{n-1},
\end{equation}
where $a_n,b_n,c_n$ are positive for all $n\ge 1$.
Let
$$\la_n:=\frac{b_n+\sqrt{b_n^2+4a_nc_n}}{2a_n}$$
be the positive root of $a_n\la^2-b_n\la-c_n=0$.
Suppose that $z_0,z_1,z_2,z_3$ is log-convex.
If there exists a sequence $\{\nu_n\}_{n\ge 1}$ of positive numbers such that $\nu_n\le\la_n$ and
\begin{equation}\label{la-c}
\Delta_n(\nu):=a_n\nu_{n-1}\nu_{n+1}-b_n\nu_{n-1}-c_n\ge 0
\end{equation}
for $n\ge 2$,
then the sequence $\{z_n\}_{n\ge 0}$ is log-convex.
\end{prop}
\begin{proof}
In Liu and Wang~\cite[Theorem 3.1]{LW07},
it is shown that if $\Delta(\la)\ge 0$, then $\{z_n\}_{n\ge 0}$ is log-convex.
So it suffices to show that $\Delta(\nu)\ge 0$ implies $\Delta(\la)\ge 0$.

Indeed, if $\Delta(\nu)\ge 0$, then $(a_n\nu_{n+1}-b_n)\nu_{n-1}\ge c_n$,
which implies that $a_n\nu_{n+1}-b_n\ge 0$.
Thus $a_n\la_{n+1}-b_n\ge 0$ and $(a_n\la_{n+1}-b_n)\la_{n-1}\ge (a_n\nu_{n+1}-b_n)\nu_{n-1}\ge c_n$,
and so $\Delta(\la)\ge 0$, as required.
\end{proof}

The $n$th trinomial coefficient $T_n$ is the coefficient of $x^n$ in the expansion $(1+x+x^2)^n$.
It is known that
\begin{equation}\label{tn-rr}
(n+1)T_{n+1}=(2n+1)T_n+3nT_{n-1}
\end{equation}
and
$$\{T_{n}\}_{n\ge 0}=\{1,1,3,7,19,51,141,393,\ldots\}.\qquad \cite[A002426]{OEIS}$$
We have the following result, which was conjectured by Sun~\cite[Conjecture 3.6]{Sun-conj}

\begin{cor}
%The sequence $\{T_n\}_{n\ge 4}$ is log-convex and
The sequence $\{\sqrt[n]{T_n}\}_{n\ge 1}$ is strictly increasing.
\end{cor}

\begin{proof}
We first apply Proposition~\ref{lw-3.1} to prove the log-convexity of the sequence $\{T_n\}_{n\ge 4}$.
It is easy to verify that $T_4,T_5,T_6,T_7$ is log-convex. Note that
$$\la_n=\frac{2n+1+\sqrt{16n^2+16n+1}}{2(n+1)}=1+\frac{\sqrt{16n^2+16n+1}-1}{2(n+1)}=1+\frac{8n}{\sqrt{16n^2+16n+1}+1}$$
and $\sqrt{16n^2+16n+1}\le 4n+2$.
Hence
$$\la_n\ge 1+\frac{8n}{4n+3}=\frac{12n+3}{4n+3}.$$
Let $\nu_n=(12n+3)/(4n+3)$. Then for $n\ge 2$,
$$\Delta_n=(n+1)\frac{(12n-9)(12n+15)}{(4n-1)(4n+7)}-(2n+1)\frac{12n-9}{4n-1}-3n=\frac{36(n-2)}{(4n-1)(4n+7)}\ge 0.$$
Thus $\{T_n\}_{n\ge 4}$ is log-convex by Proposition~\ref{lw-3.1}.
Now $\sqrt[4]{19}<\sqrt[5]{51}$ since $19^5=2476099<51^4=6765201$.
It follows that $\{\sqrt[n]{T_n}\}_{n\ge 4}$ is strictly increasing by Theorem~\ref{thm+inc}~(iii).
Clearly, $\{\sqrt[n]{T_n}\}_{1\le n\le 4}$ is strictly increasing,
so is the total sequence $\{\sqrt[n]{T_n}\}_{n\ge 1}$.
\end{proof}

We refer the reader to \cite{Dos10} for another proof of the log-convexity of $\{T_n\}_{n\ge 4}$.

The derangements number $d_n$ counts the number of permutations of $n$ elements with no fixed points.
It is known that $d_{n+1}=nd_n+nd_{n-1}$ and
$$\{d_n\}_{n\ge 0}=\{1,0,1,2,9,44,265,1854,\ldots\}.\qquad \cite[A000166]{OEIS}$$
The Motzkin number $M_n$ counts the number of lattice paths starting from $(0,0)$ to $(n,0)$,
with steps $(1,0), (1,1)$ and $(1,-1)$, and never falling below the $x$-axis.
It is known that $(n+3)M_{n+1}=(2n+3)M_n+3nM_{n-1}$ and
$$\{M_n\}_{n\ge 0}=\{1,1,2,4,9,21,51,127,\ldots\}.\qquad \cite[A001006]{OEIS}$$
The (large) Schr\"oder number $S_n$ counts the number of king walks,
from $(0,0)$ to $(n,n)$, and never rising above the line $y=x$.
It is known that $(n+2)S_{n+1}=3(2n+1)S_n-(n-1)S_{n-1}$ and
$$\{S_n\}_{n\ge 0}=\{1,2,6,22,90,394,1806,\ldots\}.\qquad \cite[A006318]{OEIS}$$
It is shown~\cite[\S 3]{LW07} by means of recurrence relations that
three sequences $\{d_n\}_{n\ge 2},\{M_n\}_{n\ge 0}$ and $\{S_n\}_{n\ge 0}$ are log-convex respectively.
So we have the following result, which was conjectured by Sun~\cite[Conjectures 3.3, 3.7 and 3.11]{Sun-conj}.

\begin{cor}
Three sequences $\{\sqrt[n]{d_n}\}_{n\ge 2}, \{\sqrt[n]{M_n}\}_{n\ge 1}$ and $\{\sqrt[n]{S_n}\}_{n\ge 1}$ are strictly increasing respectively.
\end{cor}

Davenport and P\'olya~\cite{DP49} showed that the binomial convolution preserves log-convexity:
if both $\{x_n\}_{n\ge 0}$ and $\{y_n\}_{n\ge 0}$ are log-convex,
then so is the sequence $\{z_n\}_{n\ge 0}$ defined by
$$z_n=\sum_{k=0}^{n}\binom{n}{k}x_ky_{n-k},\qquad n=0,1,2,\ldots.$$
Let $\{a(n,k)\}_{0\le k\le n}$ be a triangle of nonnegative numbers.
A general problem is in which case the operator $z_n=\sum_{k=0}^na(n,k)x_ky_{n-k}$ preserves log-convexity.
Wang and Yeh~\cite{WY07} developed techniques to deal with such a problem for log-concavity.
For example, if the triangle $\{a(n,k)\}$ has the LC-positivity property and $a(n,k)=a(n,n-k)$,
then $z_n=\sum_{k=0}^na(n,k)x_ky_{n-k}$ preserves log-concavity.
There is a similar result for log-convexity.
The following result follows from Liu and Wang~\cite[Conjecture 5.3]{LW07},
which has been shown by Chen {\it et al.}~\cite{CTWY10}.
For the sake of brevity we here omit the details of the proof.

\begin{prop}\label{nk-square}
If both $\{x_n\}_{n\ge 0}$ and $\{y_n\}_{n\ge 0}$ are log-convex,
then so is the sequence $\{z_n\}_{n\ge 0}$ defined by
$$z_n=\sum_{k=0}^{n}\binom{n}{k}^2x_ky_{n-k},\qquad n=0,1,2,\ldots.$$
\end{prop}

Now we apply Proposition~\ref{nk-square} to verify some conjectures of Sun.
Let $g_n=\sum_{k=0}^n\binom{n}{k}^2\binom{2k}{k}$.
Then
$$\{g_n\}_{n\ge 0}=\{1,3,15,93,639,4653,35169,\ldots\}.\qquad \cite[A002893]{OEIS}$$
Let $D(n)=\sum_{k=0}^n\binom{n}{k}^2\binom{2k}{k}\binom{2(n-k)}{n-k}$ be the Domb numbers. Then
$$\{D(n)\}_{n\ge 0}=\{1,4,28,256,2716,31504,\ldots\}.\qquad \cite[A002895]{OEIS}$$
Clearly, the center binomial coefficients $\binom{2k}{k}$ is log-convex in $k$
(see \cite{LW07} for instance).
So the sequences $\{g_n\}_{n\ge 0}$ and $\{D(n)\}_{n\ge 0}$ are log-convex respectively by Proposition~\ref{nk-square}.
Thus we have the following result,
which was conjectured by Sun~\cite[Conjectures 3.9 and 3.12]{Sun-conj}.

\begin{cor}
The sequences $\{\sqrt[n]{g_n}\}_{n\ge 1}, \{D(n+1)/D(n)\}_{n\ge 0}$ and $\{\sqrt[n]{D(n)}\}_{n\ge 1}$ are strictly increasing respectively.
\end{cor}

Another main result of this paper is the following criterion for log-convexity.

\begin{thm}\label{thm+sum}
Suppose that
$$z_n=\sum_{k\ge 1}\frac{\alpha_k}{\la_{k}^{n}}\qquad n=0,1,2,\ldots,$$
where $\{\alpha_k\}_{k\ge 1}, \{\la_k\}_{k\ge 1}$ are two nonnegative sequences and $\la_k$ is not constant.
Then the sequence $\{z_n\}_{n\ge 0}$ is log-convex.
\end{thm}
\begin{proof}
We have
\begin{eqnarray*}
z_{n+1}z_{n-1}-z_n^2&=&\sum_{k\ge 1}\frac{\alpha_k}{\la_{k}^{n+1}}\sum_{k\ge 1}\frac{\alpha_k}{\la_{k}^{n-1}}
-\sum_{k\ge 1}\frac{\alpha_k}{\la_{k}^{n}}\sum_{k\ge 1}\frac{\alpha_k}{\la_{k}^{n}}\\
&=&\sum_{j>i\ge 1}\frac{\alpha_i\alpha_j(\la_{i}^2+\la_{j}^2-2\la_{i}\la_{j})}{\la_{i}^{n+1}\la_{j}^{n+1}}\\
&=&\sum_{j>i\ge 1}\frac{\alpha_i\alpha_j(\la_{i}-\la_{j})^2}{\la_{i}^{n+1}\la_{j}^{n+1}}.
\end{eqnarray*}
Thus $z_{n+1}z_{n-1}-z_n^2\ge 0$,
and the sequence $\{z_n\}_{n\ge 0}$ is therefore log-convex.
\end{proof}

%We next prove Theorem~\ref{thm+sum} and then apply it to verify some conjectures of Sun and of Amdeberhan {\it et al.}~\cite{AMV12}.
Taking $\la_k=k$,
then $z_n$ is precisely the Dirichlet generating function of the sequence $\{\alpha_k\}_{k\ge 1}$.
In particular, $z_n$ coincides with Riemann zeta function $\zeta(n)$ when $\alpha_k=1$ for all $k$.
Thus the sequence $\left\{\zeta(n)\right\}_{n\ge 1}$ is strictly log-convex.
On the other hand, taking $\la_k=k^2$ and $\alpha_k=1$ for all $k$,
then the sequence $\left\{\zeta(2n)\right\}_{n\ge 1}$ is also strictly log-convex.
These two results have been obtained by Chen {\it et al.}~\cite{CGW} in an analytical approach.

The classical Bernoulli numbers are defined by
$$B_0=1,\quad \sum_{k=0}^n\binom{n+1}{k}B_k=0,\qquad n=1,2,\ldots.$$
It is well known that $B_{2n+1}=0, (-1)^{n-1}B_{2n}>0$ for $n\ge 1$ and
$$(-1)^{n-1}B_{2n}=\frac{2(2n)!\zeta(2n)}{(2\pi)^{2n}}$$
(see \cite[(6.89)]{GKP94} for instance).
It immediately follows that the sequence $\{(-1)^{n-1}B_{2n}\}_{n\ge 1}$ is log-convex,
and the sequence $\{\sqrt[n]{(-1)^{n-1}B_{2n}}\}_{n\ge 1}$ is therefore strictly increasing,
which was conjectured by Sun~\cite[Conjecture 2.15]{Sun-conj}
and has been verified by Chen {\it et al.}~\cite{CGW} and by Luca and St\u{a}nic\u{a}~\cite{LS12} respectively.

Now consider the tangent numbers
$$\{T(n)\}_{n\ge 0}=\{1, 2, 16, 272, 7936, 353792,\ldots\},\qquad \cite[A000182]{OEIS}$$
which are defined by
$$\tan{ x}=\sum_{n\geq1}T(n)\frac{x^{2n-1}}{(2n-1)!}$$
and are closely related to the Bernoulli numbers:
$$T(n)=(-1)^{n-1}B_{2n}\frac{(4^{n}-1)}{2n}4^{n}$$
(see \cite[(6.93)]{GKP94} for instance). It is not difficult to verify that $(4^{n}-1)/n$ is log-convex in $n$
(we leave the details to the reader).
On the other hand, the product of log-convex sequences is still log-convex.
So the sequence $\{T(n)\}_{n\ge 0}$ is log-convex.
Thus we have the following result,
which was conjectured by Sun~\cite[Conjecture 3.5]{Sun-conj}.

\begin{cor}
Both $\left\{T(n+1)/T(n)\right\}_{n\ge 0}$ and $\{\sqrt[n]{T(n)}\}_{n\ge 1}$ are strictly increasing.
\end{cor}

Let $A_n$ be defined by the recurrence relation
$$(-1)^{n-1}A_n=C_n+\sum_{j=1}^{n-1}(-1)^j\binom{2n-1}{2j-1}A_jC_{n-j},$$
with $A_1=1$ and $C_n=\frac{1}{n+1}\binom{2n}{n}$ the Catalan number.
Let $a_n=2A_n/C_n$.
Lasalle~\cite{Las12} and Amdeberhan {\it et al.}~\cite{AMV12} showed that
both $\{A_n\}_{n\ge 1}$ and $\{a_n\}_{n\ge 2}$ are increasing sequences of positive integers.
The latter also obtained the recurrence
$$2na_n=\sum_{k=1}^{n-1}\binom{n}{k-1}\binom{n}{k+1}a_ka_{n-k},\qquad n=2,3,\ldots$$
with $a_1=1$, and defined another sequence $\{b_n\}_{n\ge 1}$ by the recurrence
$$b_n=\sum_{k=1}^{n-1}\binom{n-1}{k-1}\binom{n-1}{k+1}b_kb_{n-k},\qquad n=2,3,\ldots$$
with $b_1=1$.
They~\cite[Conjecture 9.1]{AMV12} conjectured that both $\{a_n\}_{n\ge 1}$ and $\{b_n\}_{n\ge 1}$ are log-convex.

Let $\{j_{\mu,k}\}_{k\ge 1}$ be the (nonzero) zeros of the Bessel function of the first kind
$$J_{\mu}(x)=\sum_{m=0}^{\infty}\frac{(-1)^m}{m!\Gamma(m+\mu+1)}\left(\frac{x}{2}\right)^{2m+\mu}$$
and let
$$\zeta_{\mu}(s)=\sum_{k=1}^{\infty}\frac{1}{j_{\mu,k}^s}$$
be the Bessel zeta function. Then
\begin{eqnarray*}
A_n&=&2^{2n+1}(2n-1)!\zeta_1(2n),\\
a_n&=&2^{2n+1}(n+1)!(n-1)!\zeta_1(2n),\\
b_n&=&2^{2n-1}(n-1)!n!\zeta_{0}(2n).
\end{eqnarray*}
See \cite[Corollary 5.3, 5.4 and (7.14)]{AMV12} for details.
Now note that both $J_{0}(x)$ and $J_{1}(x)$ have only real zeros.
Hence both sequences $\{\zeta_1(2n)\}_{n\ge 0}$ and $\{\zeta_0(2n)\}_{n\ge 0}$ are log-convex by Theorem~\ref{thm+sum}.
This leads to an affirmation answer to \cite[Conjecture 9.1]{AMV12}.

\begin{cor}%(\cite[Conjecture 9.1]{AMV12})
Three sequences $\{A_n\}$, $\{a_n\}$ and $\{b_n\}$ are log-convex respectively.
\end{cor}
%%%%%%%%%%%%%%%%%%%%%%%%%%%%%%%%%%%%%%%%%%%
\section{Further work}
%%%%%%%%%%%%%%%%%%%%%%%%%%%%%%%%%%%%%%%%%%%
Sun~\cite{Sun-conj} also proposed a series of conjectures
about monotonicity of sequences of the form $\{\sqrt[n+1]{z_{n+1}}/\sqrt[n]{z_n}\}$.
Roughly speaking,
he conjectured that $\{\sqrt[n+1]{z_{n+1}}/\sqrt[n]{z_n}\}$ has the reverse monotonicity to $\{\sqrt[n]{z_n}\}$
for certain number-theoretic and combinatorial sequences $\{z_n\}$.
Clearly, if $\{\sqrt[n+1]{z_{n+1}}/\sqrt[n]{z_n}\}$ is decreasing (increasing, resp.) with the limit $1$,
then $\{\sqrt[n]{z_n}\}$ is increasing (decreasing, resp.).
It is a challenging problem to study monotonicity of $\{\sqrt[n+1]{z_{n+1}}/\sqrt[n]{z_n}\}$,
which is equivalent to log-concavity and log-convexity of $\{\sqrt[n]{z_n}\}$.
A natural problem is to ask in which case the log-convexity (log-concavity, resp.) of $z_n$ implies the log-concavity (log-convexity, resp.) of $\sqrt[n]{z_n}$.
%%%%%%%%%%%%%%%%%%%%%%%%%%%%%%%%%%%%%%%%
\section*{Acknowledgement}
%%%%%%%%%%%%%%%%%%%%%%%%%%%%%%%%%%%%%%%%
We thank Prof. Z.-W. Sun and Prof. A.L.-B. Yang for bringing our attention to conjectures of Sun and Amdeberhan {\it et al.}
%%%%%%%%%%%%%%%%%%%%%%%%%%%%%%%%%%%%%%%%
%%%%%%%%%%%%%%% References
%%%%%%%%%%%%%%%%%%%%%%%%%%%%%%%%%%%%%%%%

\end{document}